\documentclass[10pt]{amsart}
\usepackage{geometry} % see geometry.pdf on how to lay out the page. There's lots.
\usepackage{kpfonts}
\usepackage{amscd}
\usepackage{amsmath}
\usepackage[all,cmtip]{xy}
\usepackage{mathrsfs}
\usepackage{amsthm}
\usepackage[pdfpagelabels]{hyperref}
\usepackage{cite}
\newtheorem{theorem}{Theorem}           
\newtheorem{lemma}{Lemma}               
\newtheorem{corollary}{Corollary}

\theoremstyle{definition}

\newtheorem{definition}{Definition}

\newtheorem{remark}{Remark}

\begin{document}

\title[ON HUANG AND NOOR'S OPEN PROBLEM] {ON HUANG AND NOOR'S OPEN PROBLEM}
 \author[FA\.{I}K G\"{U}RSOY ]{FA\.{I}K G\"{U}RSOY}
\address{Faculty of Arts and Sciences, Department of Mathematics, Adiyaman University, Adiyaman, 02040, Turkey.}

\email{faikgursoy02@hotmail.com}


\begin{abstract}
In this paper, we introduce and analyze a new iterative algorithm for
solving a class of variational inclusions involving $H$-monotone operators.
The strong convergence of this algorithm is proved and estimate of its
convergence rate is supplied as an answer to an open question posed by Huang
and Noor [Huang, Z., Noor, M. A. (2007). Equivalency of convergence between
one-step iteration algorithm and two-step iteration algorithm of variational
inclusions for H-monotone mappings. Computers \& Mathematics with
Applications, 53(10), 1567-1571.].
\end{abstract}

\maketitle

\section{\protect\bigskip Introduction}

It is now considered indisputable that variational inequalities and
inclusions problems are among the most important mathematical problems. Such problems arise in many branches of
science including economics and transportation equilibrium, physics, nonlinear programming, optimization and control,
mechanics, engineering
and many others, and thus, they have been examined in detail for many years, see
[1-18, 20, 21, 23]. In the theory of variational inequalities and inclusions, the
development of influential and enforceable iterative algorithms are of the
greatest importance. In order to obtain solutions to the variational inequalities and
inclusions, numerous iterative algorithms have been introduced by a wide
audience of researchers, see [1, 2, 6-8, 10-14, 16-18, 20, 21] and the references
therein.

In \cite{Huang1}, a class of $H$-monotone
operators and a one-step iterative algorithm are introduced. It was shown in \cite{Huang1}
that this one-step iterative algorithm strongly converges to the solution of
variational inclusions for $H$-monotone and Lipschitzian continuous
operators. Two years later, Zeng et al. \cite{ZGY} generalized the result in 
\cite{Huang1} via introducing a two-step iterative algorithm in order to
solve the same problem of variational inclusion.

In 2007, Huang and Noor \cite{Noor1} showed that there is an equivalence
among convergence of some iterative algorithms including those defined in 
\cite{Huang1} and \cite{ZGY}. They also showed that, in solving the same
problem considered in \cite{Huang1} and \cite{ZGY}, the iterative algorithm
of \cite{Huang1} converges at a rate faster than the iterative algorithm of 
\cite{ZGY} does. In the same paper, they addressed the following question:

\textbf{Open Question (Huang and Noor \cite{Noor1}).} In order for solving the variational inclusion problem studied by \cite{Huang1} and \cite{ZGY},  can one construct a
better algorithm which converges by a more efficient convergence rate than the
algorithm of \cite{Huang1}?

In this paper, we introduce a new iterative algorithm which can be used to
approximate the solution of variational inclusions for $H$-monotone and
Lipschitzian continuous operators. Furthermore, as an affirmative answer to
the above open question, it will be shown that the iterative algorithm in
question is equivalent and faster than those introduced in \cite{Huang1} and 
\cite{ZGY}.

As a background to our exposition, we now recapitulate some definitions and
known results in the existing literature.

We will denote the set of all positive integers including zero by $\mathbb{N}$ and the set of all nonnegative real numbers by $\mathbb{R}^{+}\cup\left\{0\right\}$ over the course of this paper. Let $\mathcal{H}$ be a Hilbert space, $H:%
\mathcal{H\rightarrow H}$, $\left\langle \cdot ,\cdot \right\rangle $ be the
inner product, $\left\Vert \cdot \right\Vert $ be the norm.  $2^{\mathcal{%
H}}$ stands for family of all nonempty subsets of $\mathcal{H}$.

\begin{definition}
Let $T$, $H:\mathcal{H\rightarrow H}$ be two single-valued operators. $T$ is
said to be\newline
(i) monotone if
\begin{equation}
\left\langle Tx-Ty,x-y\right\rangle \geq 0\text{, for all }x\text{, }y\in 
\mathcal{H}\text{;}  \label{eqn1}
\end{equation}
(ii) strictly monotone if $T$ is monotone and
\begin{equation}
\left\langle Tx-Ty,x-y\right\rangle =0\text{ iff }x=y\text{;}  \label{eqn2}
\end{equation}
(iii) strongly monotone if there exists some constant $r>0$ such that
\begin{equation}
\left\langle Tx-Ty,x-y\right\rangle \geq r\left\Vert x-y\right\Vert ^{2}
\text{, for all }x\text{, }y\in \mathcal{H}\text{;}  \label{eqn3}
\end{equation}
(iv) strongly monotone w.r.t. $H$ if there exists some constant $\gamma >0$
such that
\begin{equation}
\left\langle Tx-Ty,Hx-Hy\right\rangle \geq \gamma \left\Vert x-y\right\Vert
^{2}\text{, for all }x\text{, }y\in \mathcal{H}\text{;}  \label{eqn4}
\end{equation}
(v) Lipschitz continouos if there exists some constant $s>0$ such that
\begin{equation}
\left\Vert Tx-Ty\right\Vert \leq s\left\Vert x-y\right\Vert \text{, for all }
x\text{, }y\in \mathcal{H}\text{.}  \label{eqn5}
\end{equation}
\end{definition}

\begin{remark}
(See \cite{Huang1}) $T$ and $H$ are $\tau $ and $s$-Lipschitz continuous,
repectively, and $T$ is strongly monotone w.r.t. $H$ with constant $\gamma $
, then $\gamma \leq \tau s$.
\end{remark}

\begin{definition}
(See \cite{ZGY}) Let $M:\mathcal{H\rightarrow }2^{\mathcal{H}}$ be a
multivalued-operator. $M$ is said to be\newline
(i) monotone if
\begin{equation}
\left\langle x-y,u-v\right\rangle \geq 0\text{, for all }u\text{, }v\in 
\mathcal{H}\text{, }x\in Mu\text{, }y\in Mv\text{;}  \label{eqn6}
\end{equation}
(ii) strongly monotone if there exists some constant $\eta >0$ such that
\begin{equation}
\left\langle x-y,u-v\right\rangle \geq \eta \left\Vert u-v\right\Vert ^{2}%
\text{, for all }u\text{, }v\in \mathcal{H}\text{, }x\in Mu\text{, }y\in Mv%
\text{;}  \label{eqn7}
\end{equation}
(iii) maximal monotone if $M$ is monotone and $\left( I+\lambda M\right)
\left( \mathcal{H}\right) =\mathcal{H}$ holds for all $\lambda >0$, where $I$
denotes the identity mapping on $H$;

(iv) maximal strongly monotone if $M$ is strongly monotone and $\left(
I+\lambda M\right) \left( \mathcal{H}\right) =\mathcal{H}$ holds for all $
\lambda >0$.
\end{definition}

\begin{remark}
(See \cite{ZGY}) Let $M:\mathcal{H\rightarrow }2^{\mathcal{H}}$ be a
multivalued-operator and denote graph of $M$ by Gr$\left( M\right) =\left\{
\left( u,x\right) \in \mathcal{H}\times \mathcal{H}:x\in Mu\right\} $. $M$
is maximal monotone if and only if $M$ is monotone and there is no other
monotone operator whose graph contains strictly the graph Gr$\left( M\right) 
$.
\end{remark}

The class of $H$-monotone operators above mentioned is defined as follows:

\begin{definition}
(See \cite{Huang1}) Let $H:\mathcal{H\rightarrow H}$ be a single-valued
operator and $M:\mathcal{H\rightarrow }2^{\mathcal{H}}$ a
multivalued-operator. $M$ is said to be

(i) $H$-monotone if $M$ is monotone and $\left( H+\lambda M\right) \left( 
\mathcal{H}\right) =\mathcal{H}$ holds for all $\lambda >0$;

(ii) strongly $H$-monotone if $M$ is strongly monotone and $\left( H+\lambda
M\right) \left( \mathcal{H}\right) =\mathcal{H}$ holds for all $\lambda >0$.
\end{definition}

\begin{corollary}
(See \cite{ZGY}) Let $H:\mathcal{H\rightarrow H}$ be a continuous and
strongly monotone single-valued operator and $M:\mathcal{H\rightarrow }2^{
\mathcal{H}}$ a multivalued-operator. Then $M$ is strongly $H$-monotone if
and only if $M$ is maximal strongly monotone.
\end{corollary}

\begin{lemma}
(See \cite{ZGY}) Let $H:\mathcal{H\rightarrow H}$ be a continuous and
strongly monotone operator with constant $\gamma >0$, and $M:\mathcal{
H\rightarrow }2^{\mathcal{H}}$ a strongly $H$-monotone multivalued-operator
with constant $\eta >0$. Then the resolvent operator $R_{M,\lambda }^{H}:
\mathcal{H\rightarrow H}$ is defined by 
\begin{equation}
R_{M,\lambda }^{H}\left( u\right) =\left( H+\lambda M\right) ^{-1}\left(
u\right) \text{, for all }u\in \emph{~}\mathcal{H}\text{ and for some }
\lambda >0\text{,}  \label{eqn8}
\end{equation}
is $\left( \frac{1}{\gamma +\lambda \eta }\right) $-Lipschitzian continuous,
i.e.,
\begin{equation}
\left\Vert R_{M,\lambda }^{H}\left( u\right) -R_{M,\lambda }^{H}\left(
v\right) \right\Vert \leq \left( \frac{1}{\gamma +\lambda \eta }\right)
\left\Vert u-v\right\Vert \text{, for all }u\text{, }v\in \emph{~}\mathcal{H}%
\text{.}  \label{eqn9}
\end{equation}
\end{lemma}

Now let us consider the following general variational inclusion problem
which is discussed earlier in \cite{Huang1} and \cite{ZGY}: find $u\in 
\mathcal{H}$, such that
\begin{equation}
0\in A\left( u\right) +M\left( u\right) \text{,}  \label{eqn10}
\end{equation}
where $A$ and $H:\mathcal{H\rightarrow H}$ are two single-valued operator
and $M:\mathcal{H\rightarrow }2^{\mathcal{H}}$ is a multivalued-operator.

The following iterative algorithms in \cite{Huang1} and \cite{ZGY} are used to solve (10), respectively.

\textbf{Algorithm FH} (See \cite{Huang1}). The iterative sequence $\left\{
u_{n}\right\} $ $\subset \mathcal{H}$defined by
\begin{equation}
\left\{ 
\begin{array}{c}
u_{0}\in \mathcal{H}\text{, \ \ \ \ \ \ \ \ \ \ \ \ \ \ \ \ \ \ \ \ \ \ \ \
\ \ \ \ \ \ \ \ \ \ \ \ \ \ \ \ \ \ \ \ \ } \\ 
u_{n+1}=R_{M\text{,}\lambda }^{H}\left[ Hu_{n}-\lambda Au_{n}\right] \text{,
for all }n\in\mathbb{N}\text{.}
\end{array}
\right.  \label{eqn11}
\end{equation}
\textbf{Algorithm ZGY }(See \cite{ZGY}) The iterative sequence $\left\{q_{n}\right\}$ $\subset \mathcal{H}$ defined by 
\begin{equation}
\left\{ 
\begin{array}{c}
q_{0}\in \mathcal{H}\text{, \ \ \ \ \ \ \ \ \ \ \ \ \ \ \ \ \ \ \ \ \ \ \ \
\ \ \ \ \ \ \ \ \ \ \ \ \ \ \ \ \ \ \ \ \ \ \ \ \ \ \ \ \ \ \ \ \ \ \ \ \ \
\ \ \ \ \ \ \ \ } \\ 
q_{n+1}=\left( 1-\xi_{n}\right) q_{n}+\xi_{n}R_{M,\lambda }^{H}\left[
Hr_{n}-\lambda Ar_{n}\right] \text{, \ \ \ \ \ \ \ \ \ \ \ \ \ \ \ \ \ \ \ \
\ \ } \\ 
r_{n}=\left( 1-\mu_{n}\right) q_{n}+\mu_{n}R_{M,\lambda }^{H}\left[
Hq_{n}-\lambda Aq_{n}\right] \text{, for all }n\in\mathbb{N}\text{,}
\end{array}\right.  \label{eqn12}
\end{equation}
where $\left\{\xi_{n}\right\} $ and $\left\{\mu_{n}\right\}$ are real
sequences in $\left[ 0,1\right]$ satisfying certain control conditions.

\begin{lemma}
(See \cite{ZGY}) Let $H:\mathcal{H\rightarrow H}$ be a strongly monotone and
Lipschitz continuous operator with constants $\gamma >0$ and $L>0$,
respectively, $A:\mathcal{H\rightarrow H}$ a Lipschitz continuous and
strongly monotone w.r.t. $H$ with constants $s>0$ and $r>0$, respectively. Let $\left\{u_{n}\right\}$, 
$\left\{q_{n}\right\}$ be two iterative sequences defined by (11) and (12) with real sequences $\left\{\xi_{n}\right\}$ and $\left\{ \mu_{n}\right\}$ $\subset \left[0,1\right]$ satisfying the condition $\sum\nolimits_{n=0}^{\infty}\xi_{n}=\infty$. 
Suppose that $M:\mathcal{H\rightarrow }2^{\mathcal{H}}$ is a strongly $H$
-monotone operator with constant $\eta >0$. If there exists a constant $\lambda >0$ such that
\begin{equation}
\kappa=\sqrt{L^{2}-2\lambda r+\lambda ^{2}s^{2}}/\left( \gamma +\lambda \eta
\right) <1\text{,}  \label{eqn13}
\end{equation}
then the following properties hold:

(i) If we define
\begin{equation}
F\left( x\right) =R_{M\text{,}\lambda }^{H}\left[ Hx-\lambda Ax\right] \text{
, for all }x\in \mathcal{H}\text{,}  \label{eqn14}
\end{equation}
then it follows that
\begin{eqnarray}
\left\Vert F\left( x\right) -F\left( y\right) \right\Vert &=&\left\Vert R_{M
\text{,}\lambda }^{H}\left[ Hx-\lambda Ax\right] -R_{M\text{,}\lambda }^{H}
\left[ Hy-\lambda Ay\right] \right\Vert  \label{eqn15} \\
&\leq &\kappa\left\Vert x-y\right\Vert \text{, for all }x\text{, }y\in \mathcal{H}
\text{,}  \notag
\end{eqnarray}
where $\kappa$ satisfies (13);\\

(ii) $x^{\ast }\in \mathcal{H}$ is a unique solution to the problem (10);

(iii) $\left\{ u_{n}\right\}$ converges strongly to $x^{\ast }\in \mathcal{H}$;

(iv) An iterative sequence $\left\{\nu_{n}\right\} $ defined by
\begin{equation}
\left\{ 
\begin{array}{c}
\nu_{0}\in \mathcal{H}\text{,  \ \ \ \ \ \ \ \ \ \ \ \ \ \ \ \ \ \ \ \ \ \ \ \ \ \ \ \ \ \ \ \ \ \ \ \ \ \ \ \ \ \ \ \ \ \ \ \ \ \ \ \ \ \ \ \ \ \ \ \ \ \ \ \ \ \ \ \ \ \ \ \ \ \ \ \ }
\\ 
\nu_{n+1}=\left( 1-\xi_{n}\right) \nu_{n}+\xi_{n}R_{M\text{,}\lambda }^{H}\left[H\nu_{n}-\lambda A\nu_{n}\right] \text{, for all }n\in\mathbb{N}\text{.}
\end{array}
\right.  \label{eqn16}
\end{equation}
converges strongly to $x^{\ast }\in \mathcal{H}$;

(v) $\left\{q_{n}\right\}$ converges strongly to $x^{\ast }\in \mathcal{H}$.
\end{lemma}

For brevity's sake, the term converges hereinafter refers to converges strongly.\\
The following lemma states the equivalence of convergences for the iterative algorithms given in \cite{Huang1} and \cite{ZGY}. 
\begin{lemma}
(See \cite{Noor1}) Let $H$, $M$, $A$, $\kappa$ and $x^{\ast}$ be as in Lemma 2 and let $\left\{ u_{n}\right\}$, $\left\{q_{n}\right\}$, $\left\{\nu_{n}\right\}$ be the iterative sequences defined by (11), (12) and (16), respectively, with real sequences $\left\{\xi_{n}\right\} $, $\left\{\mu_{n}\right\} $ in $\left[
0,1\right] $ satisfying the conditions $\lim_{n\rightarrow \infty }\xi_{n}=0$
and $\sum_{n=0}^{\infty }\xi_{n}=\infty $. Then the following are
equivalent:\newline

(i) $\left\{u_{n}\right\}$ converges to $x^{\ast }\in \mathcal{H}$;

(ii) $\left\{q_{n}\right\}$ converges to $x^{\ast }\in \mathcal{H}$;

(iii) $\left\{\nu_{n}\right\}$ converges to $x^{\ast }\in \mathcal{H}$.
\end{lemma}

Inspired by the performance and achievements of the above-mentioned
iterative algorithms, we now introduce our new iterative algorithm as follows:

\textbf{Algorithm }The iterative sequence $\left\{s_{n}\right\} $ $\subset 
\mathcal{H}$ defined by

\begin{equation}
\left\{ 
\begin{array}{c}
s_{0}\in \mathcal{H}\text{, \ \ \ \ \ \ \ \ \ \ \ \ \ \ \ \ \ \ \ \ \ \ \ \
\ \ \ \ \ \ \ \ \ \ \ \ \ \ \ \ \ \ \ \ \ \ \ \ \ \ \ \ \ \ \ \ \ \ \ \ \ \
\ \ \ \ \ \ \ } \\ 
s_{n+1}=R_{M,\lambda }^{H}\left[ Ht_{n}-\lambda At_{n}\right] \text{, \ \ \
\ \ \ \ \ \ \ \ \ \ \ \ \ \ \ \ \ \ \ \ \ \ \ \ \ \ \ \ \ \ \ \ \ \ \ \ \ \
\ \ } \\ 
t_{n}=\left( 1-\mu_{n}\right) s_{n}+\mu_{n}R_{M,\lambda }^{H}\left[
Hs_{n}-\lambda As_{n}\right] \text{, for all }n\in\mathbb{N}\text{.}
\end{array}\right.   \label{eqn17}
\end{equation}
where $\left\{\mu_{n}\right\}$ is real sequence in $\left[ 0,1\right] $
satisfying certain control conditions.

The following definitions and lemma will be very useful in obtaining the
main results of this paper.

\begin{definition}
(See \cite{Berinde}) Suppose that for two sequences $\left\{x_{n}\right\} _{n=0}^{\infty }$ and $\left\{y_{n}\right\}
_{n=0}^{\infty }$ both converging to the same point $p$, the following
error predictions are available:
\begin{equation}
\left\Vert x_{n}-p\right\Vert \leq \alpha_{n}\text{ for all }n\in\mathbb{N}\text{,}  \label{eqn18}
\end{equation}
\begin{equation}
\left\Vert y_{n}-p\right\Vert \leq \theta_{n}\text{ for all }n\in\mathbb{N}\text{,}  \label{eqn19}
\end{equation}
where $\left\{\alpha_{n}\right\} _{n=0}^{\infty }$ and $\left\{\theta_{n}\right\} _{n=0}^{\infty }$ are two sequences of positive numbers (converging to zero). If $\left\{\alpha_{n}\right\} _{n=0}^{\infty}$ converges
at a rate faster than $\left\{\theta_{n}\right\} _{n=0}^{\infty }$, then $\left\{
x_{n}\right\} _{n=0}^{\infty}$ converges at a rate faster than $\left\{y_{n}\right\}
_{n=0}^{\infty}$ to $p$.
\end{definition}

\begin{definition}
(See \cite{Berinde}) Let $\left\{\alpha_{n}\right\} _{n=0}^{\infty }$ and $
\left\{\theta_{n}\right\} _{n=0}^{\infty }$ be two sequences of real numbers
with limits $\alpha$ and $\theta$, respectively. Assume that there exists
\begin{equation}
\underset{n\rightarrow \infty }{\lim }\frac{\left\vert \alpha_{n}-\alpha\right\vert }{
\left\vert \theta_{n}-\theta\right\vert }=\ell\text{.}  \label{eqn20}
\end{equation}

(i) If $\ell=0$, then $\left\{\alpha_{n}\right\} _{n=0}^{\infty }$
converges faster to $\alpha$ than $\left\{\theta_{n}\right\} _{n=0}^{\infty}$ to $\theta$.

(ii) If $\ell\in \left( 0,\infty\right)$, then $\left\{\alpha_{n}\right\}
_{n=0}^{\infty }$ and $\left\{\theta_{n}\right\} _{n=0}^{\infty}$ have the same
rate of convergence.
\end{definition}

\begin{lemma}
(See \cite{Weng}) Let $\left\{\sigma_{n}\right\} _{n=0}^{\infty }$ and $
\left\{ \rho _{n}\right\} _{n=0}^{\infty }$$\subset\mathbb{R}^{+}\cup\left\{0\right\}$ be two sequences.
Assume that the following inequality is satisfied:
\begin{equation}
\sigma_{n+1}\leq \left( 1-\epsilon _{n}\right)\sigma_{n}+\rho _{n}\text{,}
\label{eqn21}
\end{equation}
where $\epsilon_{n}\in\left( 0,1\right)$, for all $n\geq n_{0}$, $\sum_{n=1}^{\infty }\epsilon_{n}=\infty$, and $\rho _{n}=o\left(
\epsilon_{n}\right) $. Then $\lim_{n\rightarrow \infty }\sigma_{n}=0$.
\end{lemma}

\section{Main Results}

\begin{theorem}
Let $H:\mathcal{H\rightarrow H}$ be a strongly monotone and Lipschitz
continuous operator with constants $\gamma $ and $\tau $, respectively. Let $
A:\mathcal{H\rightarrow H}$ be Lipschitz continuous and strongly monotone
with respect to $H$ with constants $s$ and $r$, respectively. Suppose that $
M:\mathcal{H\rightarrow }2^{\mathcal{H}}$ a strongly $H$-monotone operator
with constant $\eta>0$. If there exists a constant $\lambda>0$ such that
\begin{eqnarray}
\left\vert \lambda -\frac{r+\gamma \eta }{s^{2}-\eta ^{2}}\right\vert &<&
\frac{\sqrt{\left( r+\gamma \eta \right) ^{2}-\left( s^{2}-\eta ^{2}\right)
\left( \tau ^{2}-\gamma ^{2}\right) }}{s^{2}-\eta ^{2}}\text{,}
\label{eqn22} \\
\left( r+\gamma \eta \right) ^{2} &>&\left( s^{2}-\eta ^{2}\right) \left(
\tau ^{2}-\gamma ^{2}\right) \text{, }s>\eta \text{.}  \notag
\end{eqnarray}
Then the iterative sequence $\{s_{n}\}$ defined by (17) with real sequence 
$\left\{\mu_{n}\right\} $ $\subset \left[ 0,1\right] $ satisfying $
\sum_{n=0}^{\infty }\mu_{n}=\infty$ converges to the
unique solution $x^{\ast}$ of the variational inclusion problem (1.10).
\end{theorem}

\begin{proof}
Existence and uniqueness of $x^{\ast}$come from (\cite{ZGY}, Theorem 3.1).
For completion of the proof it just will be shown that the sequence $\{s_{n}\}$ converges to $x^{\ast}$.

Utilizing (9), (14), and (17), we obtain 
\begin{eqnarray}
\left\Vert s_{n+1}-x^{\ast }\right\Vert  &=&\left\Vert R_{M,\lambda }^{H}
\left[ Ht_{n}-\lambda At_{n}\right] -R_{M,\lambda }^{H}\left[ Hx^{\ast
}-\lambda Ax^{\ast }\right] \right\Vert   \notag \\
&\leq &\frac{1}{\gamma +\lambda \eta }\left\Vert Ht_{n}-Hx^{\ast }-\lambda
\left( At_{n}-Ax^{\ast }\right) \right\Vert \text{.}  \label{eqn23}
\end{eqnarray}
It follows from the assumptions that
\begin{eqnarray}
\left\Vert Ht_{n}-Hx^{\ast }-\lambda \left( At_{n}-Ax^{\ast }\right)
\right\Vert ^{2} &=&\left\Vert Ht_{n}-Hx^{\ast }\right\Vert ^{2}-2\lambda
\left\langle At_{n}-Ax^{\ast },Ht_{n}-Hx^{\ast }\right\rangle   \notag \\
&&+\lambda ^{2}\left\Vert At_{n}-Ax^{\ast }\right\Vert ^{2}  \notag \\
&\leq &\tau ^{2}\left\Vert t_{n}-x^{\ast }\right\Vert ^{2}-2\lambda
r\left\Vert t_{n}-x^{\ast }\right\Vert ^{2}  \notag \\
&&+\lambda ^{2}s^{2}\left\Vert t_{n}-x^{\ast }\right\Vert ^{2}  \notag \\
&=&\left( \tau ^{2}-2\lambda r+\lambda ^{2}s^{2}\right) \left\Vert
t_{n}-x^{\ast }\right\Vert ^{2}\text{,}  \label{eqn24}
\end{eqnarray}
or
\begin{equation}
\left\Vert Ht_{n}-Hx^{\ast }-\lambda \left( At_{n}-Ax^{\ast }\right)
\right\Vert \leq \sqrt{\tau ^{2}-2\lambda r+\lambda ^{2}s^{2}}\left\Vert
t_{n}-x^{\ast }\right\Vert \text{.}  \label{eqn25}
\end{equation}
Using again (9), (14), and (1.17), we have
\begin{eqnarray}
\left\Vert t_{n}-x^{\ast }\right\Vert  &=&\left\Vert \left( 1-\mu
_{n}\right) \left( s_{n}-x^{\ast }\right) +\mu_{n}\left( R_{M,\lambda
}^{H}\left[ Hs_{n}-\lambda As_{n}\right] -R_{M,\lambda }^{H}\left[ Hx^{\ast
}-\lambda Ax^{\ast }\right] \right) \right\Vert   \notag \\
&\leq &\left( 1-\mu_{n}\right) \left\Vert s_{n}-x^{\ast }\right\Vert
+\mu_{n}\left\Vert R_{M,\lambda }^{H}\left[ Hs_{n}-\lambda As_{n}\right]
-R_{M,\lambda }^{H}\left[ Hx^{\ast }-\lambda Ax^{\ast }\right] \right\Vert  
\notag \\
&\leq &\left( 1-\mu_{n}\right) \left\Vert s_{n}-x^{\ast }\right\Vert
+\mu_{n}\frac{1}{\gamma +\lambda \eta }\left\Vert Hs_{n}-Hx^{\ast
}-\lambda \left( As_{n}-Ax^{\ast }\right) \right\Vert   \notag \\
&\leq &\left[ 1-\mu_{n}\left( 1-\kappa\right) \right] \left\Vert
s_{n}-x^{\ast }\right\Vert \text{,}  \label{eqn26}
\end{eqnarray}
where $\kappa=\frac{1}{\gamma +\lambda \eta }\sqrt{\tau ^{2}-2\lambda r+\lambda
^{2}s^{2}}<1$.

Combining (23), (25), and (26)
\begin{equation}
\left\Vert s_{n+1}-x^{\ast }\right\Vert \leq \kappa\left[ 1-\mu_{n}\left(
1-\kappa\right) \right] \left\Vert s_{n}-x^{\ast }\right\Vert \text{.}
\label{eqn27}
\end{equation}
Inductively
\begin{equation}
\left\Vert s_{n+1}-x^{\ast }\right\Vert \leq \left\Vert s_{0}-x^{\ast
}\right\Vert \kappa^{n+1}\prod_{i=0}^{n}\left[ 1-\mu_{i}\left(
1-\kappa\right) \right] \text{.}  \label{eqn28}
\end{equation}
It is well-known from the classical analysis that $1-a\leq e^{-a}$ for all $
a\in \left[ 0,1\right]$. Hence, (28) becomes
\begin{equation}
\left\Vert s_{n+1}-x^{\ast }\right\Vert \leq \left\Vert s_{0}-x^{\ast
}\right\Vert \kappa^{n+1}e^{-\left( 1-\kappa\right) \sum_{i=0}^{n}\mu
_{i}}\text{.}  \label{eqn29}
\end{equation}
It follows from the assumption $\sum_{i=0}^{\infty }\mu
_{i}=\infty $ that $e^{-\left( 1-\kappa\right) \sum_{i=0}^{n}\mu
_{i}}\rightarrow 0$ as $n\rightarrow \infty $, which implies $
\lim_{n\rightarrow \infty }\left\Vert s_{n}-x^{\ast }\right\Vert =0$.
\end{proof}

\begin{theorem}
Let $H$, $M$, $A$, $\kappa$, and $x^{\ast}$ be defined as in Theorem 1. Let $\{q_{n}\}$, $\{s_{n}\}$
be two iterative sequence generated by (12) and (17), respectively, with
real sequences $\{\xi_{n}\}$, $\left\{\mu_{n}\right\} $ $\subset \left[
0,1\right] $ satisfying the condition $\sum\limits_{n=0}^{\infty
}\xi_{n}\mu_{n}\left( 1-\kappa\right) =\infty $. Then the following are
equivalent:

(i) $\{q_{n}\}$ converges to $x^{\ast }\in \mathcal{H}$. 

(ii) $\{s_{n}\}$ converges to $x^{\ast }\in \mathcal{H}$. 
\end{theorem}

\begin{proof}
We will prove (i)$\Rightarrow $(ii), that is, if $
\{q_{n}\}$ converges to $x^{\ast }$, then $\{s_{n}\}$
does too.

It derives from (12), (14), (15), and (17) that
\begin{eqnarray}
\left\Vert q_{n+1}-s_{n+1}\right\Vert  &=&\left\Vert \left( 1-\xi_{n}\right)
q_{n}+\xi_{n}R_{M,\lambda }^{H}\left[ Hr_{n}-\lambda Ar_{n}\right]
-R_{M,\lambda }^{H}\left[ Ht_{n}-\lambda At_{n}\right] \right\Vert   \notag
\\
&\leq &\left( 1-\xi_{n}\right) \left\Vert q_{n}-R_{M,\lambda }^{H}\left[
Ht_{n}-\lambda At_{n}\right] \right\Vert   \notag \\
&&+\xi_{n}\left\Vert R_{M,\lambda }^{H}\left[ Hr_{n}-\lambda Ar_{n}\right]
-R_{M,\lambda }^{H}\left[ Ht_{n}-\lambda At_{n}\right] \right\Vert   \notag
\\
&=&\left( 1-\xi_{n}\right) \left\Vert q_{n}-F\left( t_{n}\right) \right\Vert
+\xi_{n}\left\Vert F\left(r_{n}\right) -F\left( t_{n}\right) \right\Vert  
\notag \\
&\leq &\left( 1-\xi_{n}\right) \left\{ \left\Vert q_{n}-x^{\ast }\right\Vert
+\left\Vert F\left( x^{\ast }\right) -F\left( t_{n}\right) \right\Vert
\right\} +\xi_{n}\left\Vert F\left(r_{n}\right) -F\left( t_{n}\right)
\right\Vert   \notag \\
&\leq &\left( 1-\xi_{n}\right) \left\{ \left\Vert q_{n}-x^{\ast }\right\Vert
+\kappa\left\Vert x^{\ast}-t_{n}\right\Vert \right\} +\xi_{n}\kappa\left\Vert
r_{n}-t_{n}\right\Vert   \notag \\
&\leq &\left( 1-\xi_{n}\right) \left\Vert q_{n}-x^{\ast }\right\Vert +\left(
1-\xi_{n}\right) \kappa\left\Vert r_{n}-x^{\ast }\right\Vert   \notag \\
&&+\left( 1-\xi_{n}\right) \kappa\left\Vert r_{n}-t_{n}\right\Vert
+\xi_{n}\kappa\left\Vert r_{n}-t_{n}\right\Vert   \notag \\
&=&\left( 1-\xi_{n}\right) \left\Vert q_{n}-x^{\ast }\right\Vert +\left(
1-\xi_{n}\right) \kappa\left\Vert r_{n}-x^{\ast }\right\Vert +\kappa\left\Vert
r_{n}-t_{n}\right\Vert \text{,}  \label{eqn30}
\end{eqnarray}
\begin{eqnarray}
\left\Vert r_{n}-x^{\ast }\right\Vert  &=&\left\Vert \left( 1-\mu
_{n}\right) q_{n}+\mu_{n}R_{M,\lambda }^{H}\left[ Hq_{n}-\lambda Aq_{n}
\right] -x^{\ast }\right\Vert   \notag \\
&\leq &\left( 1-\mu_{n}\right) \left\Vert q_{n}-x^{\ast }\right\Vert
+\mu_{n}\left\Vert R_{M,\lambda }^{H}\left[ Hq_{n}-\lambda Aq_{n}\right]
-R_{M,\lambda }^{H}\left[ Hx^{\ast }-\lambda Ax^{\ast }\right] \right\Vert  
\notag \\
&=&\left( 1-\mu_{n}\right) \left\Vert q_{n}-x^{\ast }\right\Vert +\mu
_{n}\left\Vert F\left(q_{n}\right) -F\left( x^{\ast }\right) \right\Vert  
\notag \\
&\leq &\left( 1-\mu_{n}\right) \left\Vert q_{n}-x^{\ast }\right\Vert
+\mu_{n}\kappa\left\Vert q_{n}-x^{\ast }\right\Vert   \notag \\
&=&\left[ 1-\mu_{n}\left( 1-\kappa\right) \right] \left\Vert q_{n}-x^{\ast
}\right\Vert \text{,}  \label{eqn31}
\end{eqnarray}
\begin{eqnarray}
\left\Vert r_{n}-t_{n}\right\Vert  &=&\left\Vert \left( 1-\mu_{n}\right)
\left( q_{n}-s_{n}\right) +\mu_{n}\left( R_{M,\lambda }^{H}\left[
Hq_{n}-\lambda Aq_{n}\right] -R_{M,\lambda }^{H}\left[ Hs_{n}-\lambda As_{n}
\right] \right) \right\Vert   \notag \\
&\leq &\left( 1-\mu_{n}\right) \left\Vert q_{n}-s_{n}\right\Vert +\mu
_{n}\left\Vert F\left(q_{n}\right) -F\left( s_{n}\right) \right\Vert  
\notag \\
&\leq &\left[ 1-\mu_{n}\left( 1-\kappa\right) \right] \left\Vert
q_{n}-s_{n}\right\Vert \text{.}  \label{eqn32}
\end{eqnarray}
Combining (30), (31), and (32) and using the fact $\kappa\in \left(0,1\right)$ that
\begin{eqnarray}
\left\Vert q_{n+1}-s_{n+1}\right\Vert  &\leq &\left( 1-\xi_{n}\right)
\left\Vert q_{n}-x^{\ast }\right\Vert +\left( 1-\xi_{n}\right) \kappa\left[
1-\mu_{n}\left( 1-\kappa\right) \right] \left\Vert q_{n}-x^{\ast }\right\Vert 
\notag \\
&&+k\left[ 1-\mu_{n}\left( 1-\kappa\right) \right] \left\Vert
q_{n}-s_{n}\right\Vert   \notag \\
&\leq &\left[ 1-\mu_{n}\left( 1-\kappa\right) \right] \left\Vert
q_{n}-s_{n}\right\Vert   \notag \\
&&+\left( 1-\xi_{n}\right) \left\{ 1+\kappa\left[ 1-\mu_{n}\left(1-\kappa\right)
\right] \right\} \left\Vert q_{n}-x^{\ast }\right\Vert \text{.}
\label{eqn33}
\end{eqnarray}
Denote
\begin{eqnarray*}
\sigma_{n} &=&\left\Vert q_{n}-s_{n}\right\Vert \text{,} \\
\epsilon_{n} &=&\mu_{n}\left(1-\kappa\right) \in \left( 0,1\right) \text{,} \\
\rho _{n} &=&\left( 1-\xi_{n}\right) \left\{ 1+\kappa\left[ 1-\mu_{n}\left(
1-\kappa\right) \right] \right\} \left\Vert q_{n}-x^{\ast }\right\Vert \text{.}
\end{eqnarray*}
Then (33) becomes
\begin{equation}
\sigma_{n+1}\leq \left( 1-\epsilon_{n}\right) \sigma_{n}+\rho _{n}\text{, }
n\geq 0\text{.}  \label{eqn34}
\end{equation}
As $\lim_{n\rightarrow \infty }\left\Vert q_{n}-x^{\ast }\right\Vert =0$, 
$\rho _{n}=o\left(\epsilon_{n}\right) $. Also, since $\xi_{n}$, $\mu
_{n}\in $ $\left[ 0,1\right] $ for all $n\in\mathbb{N}$ and $\kappa\in\left(0,1\right)$, we have
\begin{equation*}
\xi_{n}\mu_{n}\left( 1-\kappa\right) <\mu_{n}\left(1-\kappa\right) \text{,}
\end{equation*}
it thus follows from the comparison test for infinite series that
\begin{equation*}
\sum_{n=0}^{\infty}\xi_{n}\mu_{n}\left(1-\kappa\right)
<\sum_{n=0}^{\infty}\mu_{n}\left(1-\kappa\right) =\infty \text{.}
\end{equation*}
Hence, an application of Lemma 2 to (2.13) yields $\lim_{n\rightarrow \infty
}\sigma_{n}=\lim_{n\rightarrow \infty }\left\Vert q_{n}-s_{n}\right\Vert =0$. Since $\lim_{n\rightarrow \infty }\left\Vert q_{n}-x^{\ast }\right\Vert =0$
and 
\begin{equation*}
\left\Vert s_{n}-x^{\ast }\right\Vert \leq \left\Vert q_{n}-s_{n}\right\Vert
+\left\Vert q_{n}-x^{\ast }\right\Vert \text{,}
\end{equation*}
we have $\lim_{n\rightarrow \infty }\left\Vert s_{n}-x^{\ast }\right\Vert =0$.

Next we will prove (ii)$\Rightarrow $(i), that is, if $s_{n}\rightarrow
x^{\ast }$ as $n\rightarrow \infty $, then $q_{n}\rightarrow x^{\ast }$ as $n\rightarrow \infty $.

Utilizing (12), (14), (15), and (17), we have
\begin{eqnarray}
\left\Vert s_{n+1}-q_{n+1}\right\Vert  &=&\left\Vert R_{M,\lambda }^{H}\left[
Ht_{n}-\lambda At_{n}\right] -\left( 1-\xi_{n}\right) q_{n}-\xi_{n}R_{M,\lambda
}^{H}\left[ Hr_{n}-\lambda Ar_{n}\right] \right\Vert   \notag \\
&\leq &\left( 1-\xi_{n}\right) \left\Vert R_{M,\lambda }^{H}\left[
Ht_{n}-\lambda At_{n}\right] -q_{n}\right\Vert   \notag \\
&&+\xi_{n}\left\Vert R_{M,\lambda }^{H}\left[ Ht_{n}-\lambda At_{n}\right]
-R_{M,\lambda }^{H}\left[ Hr_{n}-\lambda Ar_{n}\right] \right\Vert   \notag
\\
&=&\left(1-\xi_{n}\right) \left\Vert F\left(t_{n}\right) -q_{n}\right\Vert
+\xi_{n}\left\Vert F\left(t_{n}\right) -F\left(r_{n}\right) \right\Vert  
\notag \\
&\leq &\left(1-\xi_{n}\right) \left\{ \left\Vert s_{n}-q_{n}\right\Vert
+\left\Vert s_{n}-F\left( x^{\ast }\right) \right\Vert +\left\Vert F\left(
t_{n}\right) -F\left( x^{\ast }\right) \right\Vert \right\}   \notag \\
&&+\xi_{n}\left\Vert F\left(r_{n}\right) -F\left(t_{n}\right) \right\Vert  
\notag \\
&\leq &\left(1-\xi_{n}\right) \left\{ \left\Vert s_{n}-q_{n}\right\Vert
+\left\Vert s_{n}-x^{\ast }\right\Vert +\kappa\left\Vert t_{n}-x^{\ast
}\right\Vert \right\}   \notag \\
&&+\xi_{n}\kappa\left\Vert t_{n}-r_{n}\right\Vert \text{,}  \label{eqn35}
\end{eqnarray}
\begin{eqnarray}
\left\Vert t_{n}-x^{\ast }\right\Vert  &=&\left\Vert \left(1-\mu
_{n}\right) s_{n}+\mu_{n}R_{M,\lambda }^{H}\left[ Hs_{n}-\lambda As_{n}
\right] -x^{\ast }\right\Vert   \notag \\
&\leq &\left(1-\mu_{n}\right) \left\Vert s_{n}-x^{\ast }\right\Vert
+\mu_{n}\left\Vert R_{M,\lambda }^{H}\left[ Hs_{n}-\lambda As_{n}\right]
-R_{M,\lambda }^{H}\left[ Hx^{\ast }-\lambda Ax^{\ast }\right] \right\Vert  
\notag \\
&=&\left(1-\mu_{n}\right) \left\Vert s_{n}-x^{\ast }\right\Vert +\mu
_{n}\left\Vert F\left( s_{n}\right) -F\left( x^{\ast }\right) \right\Vert  
\notag \\
&\leq &\left(1-\mu_{n}\right) \left\Vert s_{n}-x^{\ast }\right\Vert
+\mu_{n}\kappa\left\Vert s_{n}-x^{\ast }\right\Vert   \notag \\
&=&\left[1-\mu_{n}\left(1-\kappa\right) \right] \left\Vert s_{n}-x^{\ast
}\right\Vert \text{,}  \label{eqn36}
\end{eqnarray}
\begin{eqnarray}
\left\Vert t_{n}-r_{n}\right\Vert  &=&\left\Vert \left(1-\mu_{n}\right)
\left( s_{n}-q_{n}\right) +\mu_{n}\left( R_{M,\lambda }^{H}\left[
Hs_{n}-\lambda As_{n}\right] -R_{M,\lambda }^{H}\left[ Hq_{n}-\lambda Aq_{n}
\right] \right) \right\Vert   \notag \\
&\leq &\left(1-\mu_{n}\right) \left\Vert s_{n}-q_{n}\right\Vert +\mu
_{n}\left\Vert F\left( s_{n}\right) -F\left( q_{n}\right) \right\Vert  
\notag \\
&\leq &\left[1-\mu_{n}\left( 1-\kappa\right) \right] \left\Vert
s_{n}-q_{n}\right\Vert \text{.}  \label{eqn37}
\end{eqnarray}
Substituting (36) and (37) into (35), and using the fact $\kappa\in \left(0,1\right)$ we obtain
\begin{eqnarray}
\left\Vert s_{n+1}-q_{n+1}\right\Vert  &\leq &\left[ 1-\xi_{n}\mu
_{n}\left(1-\kappa\right) \right] \left\Vert s_{n}-q_{n}\right\Vert   \notag \\
&&+\left(1-\xi_{n}\right) \left\{1+\kappa\left[1-\mu_{n}\left(1-\kappa\right)
\right] \right\} \left\Vert s_{n}-x^{\ast }\right\Vert \text{.}
\label{eqn38}
\end{eqnarray}
Denote that
\begin{eqnarray*}
\sigma_{n} &=&\left\Vert s_{n}-q_{n}\right\Vert \text{,} \\
\epsilon_{n} &=&\xi_{n}\mu_{n}\left(1-\kappa\right) \in \left( 0,1\right) \text{,} \\
\rho _{n} &=&\left(1-\xi_{n}\right) \left\{1+\kappa\left[1-\mu_{n}\left(
1-\kappa\right) \right] \right\} \left\Vert s_{n}-x^{\ast }\right\Vert \text{.}
\end{eqnarray*}
Since $\lim_{n\rightarrow \infty }\left\Vert s_{n}-x^{\ast }\right\Vert =0$, 
$\rho _{n}=o\left(\epsilon_{n}\right) $. Hence, an application of Lemma 2 to
(2.17) yields $\lim_{n\rightarrow \infty }\sigma_{n}=\lim_{n\rightarrow
\infty }\left\Vert s_{n}-q_{n}\right\Vert =0$. As $\lim_{n\rightarrow
\infty }\left\Vert s_{n}-x^{\ast }\right\Vert =0$ and 
\begin{equation*}
\left\Vert q_{n}-x^{\ast }\right\Vert \leq \left\Vert s_{n}-q_{n}\right\Vert
+\left\Vert s_{n}-x^{\ast }\right\Vert \text{,}
\end{equation*}
we have $\lim_{n\rightarrow \infty }\left\Vert q_{n}-x^{\ast }\right\Vert =0$.
\end{proof}

Taking Lemma 3 into account, Theorem 2 leads to the following corollary.

\begin{corollary}
Let $H$, $M$, $A$, $\kappa$ and $x^{\ast}$ be as in Theorem 1 and let $\left\{ u_{n}\right\}$, $\left\{q_{n}\right\}$, $\left\{\nu_{n}\right\}$, $\left\{s_{n}\right\}$
be the iterative sequences defined by (11), (12), (16) and (17), respectively, with real sequences $\left\{\xi_{n}\right\} $, $\left\{ \mu_{n}\right\} $ in $\left[ 0,1\right]$
satisfying certain control conditions as in Lemma 3 and Theorem 2. Then the following are equivalent:

(i) $\left\{u_{n}\right\}$ converges to $x^{\ast}\in \mathcal{H}$;

(ii) $\left\{q_{n}\right\}$ converges to $x^{\ast}\in \mathcal{H}$;

(iii) $\left\{\nu_{n}\right\}$ converges to $x^{\ast}\in \mathcal{H}$;

(iv) $\left\{s_{n}\right\}$ converges to $x^{\ast }\in \mathcal{H}$.
\end{corollary}

Now we are in a position to give the following result which is of great
importance both theoretical and numerical aspects.

\begin{theorem}
Let $H$, $M$, $A$, $\kappa$, and $x^{\ast}$ be defined as in Theorem 1 and let $\left\{\mu
_{n}\right\} $ be a sequence in $\left[ 0,1\right] $ satisfying 

(i) $\sum_{n=0}^{\infty}\mu_{n}=\infty $,

(ii) $\mu \leq\mu_{n}\leq 1$,  for all $n\in\mathbb{N}$ and for some $\mu>0$.

For given $u_{0}=s_{0}\in \mathcal{H}$, consider the iterative sequences $
\left\{u_{n}\right\} $ and $\{s_{n}\}$ defined by (11) and (17),
respectively. Then $\{s_{n}\}$ converges to $x^{\ast }$ at a
rate faster than $\left\{u_{n}\right\}$ does.
\end{theorem}

\begin{proof}
The following inequality was obtained in (\cite{Noor1}, Theorem 2)
\begin{equation}
\left\Vert u_{n}-x^{\ast }\right\Vert \leq \kappa^{n}\left\Vert u_{0}-x^{\ast
}\right\Vert \text{.}  \label{eqn39}
\end{equation}
From Theorem 1, we have
\begin{equation*}
\left\Vert s_{n+1}-x^{\ast }\right\Vert \leq \kappa^{n+1}\left\Vert s_{0}-x^{\ast
}\right\Vert \prod_{i=0}^{n}\left[1-\mu_{i}\left( 1-\kappa\right)\right] \text{,}
\end{equation*}
or equivalently
\begin{equation}
\left\Vert s_{n}-x^{\ast }\right\Vert \leq \kappa^{n}\left\Vert s_{0}-x^{\ast
}\right\Vert \prod_{i=1}^{n}\left[1-\mu_{i-1}\left(
1-\kappa\right) \right] \text{.}  \label{eqn40}
\end{equation}
It follows from assumption (ii) that
\begin{equation}
\left\Vert s_{n}-x^{\ast }\right\Vert \leq \kappa^{n}\left\Vert s_{0}-x^{\ast
}\right\Vert \prod_{i=1}^{n}\left[ 1-\mu \left( 1-\kappa\right)\right] \text{.}  \label{eqn41}
\end{equation}
Denote that
\begin{eqnarray*}
\alpha_{n} &=&\kappa^{n}\left\Vert s_{0}-x^{\ast }\right\Vert \prod_{i=1}^{n}
\left[ 1-\mu \left( 1-\kappa\right) \right] \text{,} \\
\theta_{n} &=&\kappa^{n}\left\Vert u_{0}-x^{\ast }\right\Vert \text{.}
\end{eqnarray*}
Since $\lim_{n\rightarrow \infty }\kappa^{n}=0$, $\lim_{n\rightarrow \infty
}\alpha_{n}=0$ and $\lim_{n\rightarrow \infty }\theta_{n}=0$, that is, both the
sequences $\left\{\alpha_{n}\right\}$ and $\left\{\theta_{n}\right\}$  converges to zero as assumed in Definition 4. 

Define
\begin{eqnarray*}
\pi_{n} &=&\frac{\left\vert \alpha_{n}-0\right\vert }{\left\vert
\theta_{n}-0\right\vert }=\frac{\kappa^{n}\left\Vert s_{0}-x^{\ast }\right\Vert
\prod_{i=1}^{n}\left[ 1-\mu \left( 1-\kappa\right) \right] }{\kappa^{n}\left\Vert u_{0}-x^{\ast }\right\Vert } \\
&=&\prod_{i=1}^{n}\left[ 1-\mu \left( 1-\kappa\right) \right] \text{.}
\end{eqnarray*}
Thus, we have
\begin{eqnarray*}
\lim_{n\rightarrow \infty }\frac{\pi_{n+1}}{\pi_{n}}
&=&\lim_{n\rightarrow \infty }\frac{\prod_{i=1}^{n+1}\left[
1-\mu \left( 1-\kappa\right) \right] }{\prod_{i=1}^{n}\left[
1-\mu \left( 1-\kappa\right) \right] } \\
&=&\lim_{n\rightarrow \infty }\left[ 1-\mu \left( 1-\kappa\right) \right]  \\
&=&1-\mu \left( 1-\kappa\right) <1\text{,}
\end{eqnarray*}
and so by the ratio test the series $\sum_{n=0}^{\infty}\pi
_{n}$ converges absolutely and hence will converge. This allows us to
conclude that 
\begin{equation*}
\lim_{n\rightarrow \infty }\pi_{n}=\lim_{n\rightarrow \infty }\frac{
\left\vert \alpha_{n}-0\right\vert }{\left\vert \theta_{n}-0\right\vert }=0\text{.}
\end{equation*}
Having regard to part (i) of Definition 5, we conclude that $\left\{ \alpha_{n}\right\} $ converges faster than $\left\{
\theta_{n}\right\} $ which implies that $\left\{ s_{n}\right\}$ converges faster than $\left\{ u_{n}\right\}$.
\end{proof}

\end{document}